\documentclass[12pt]{article}
\usepackage[usenames]{color}
\usepackage{graphicx, subfigure}
\usepackage{amsmath, amsthm, amssymb}
\usepackage{amsfonts}
\usepackage{fullpage}
\usepackage{ifthen}
\usepackage{url}
\usepackage[sort&compress]{natbib}
\usepackage{multirow}
\usepackage{bm}
\usepackage{tikz-qtree}
\usepackage[linesnumbered,algoruled,boxed,lined,commentsnumbered]{algorithm2e}
\usepackage[section]{placeins}
\usepackage{soul}
\usepackage{xcolor,cancel}
\usepackage{mathtools}

\newtheorem{theorem}{Theorem}[section]

\makeatletter
\def\@biblabel#1{}
\makeatother

\theoremstyle{plain}

\theoremstyle{definition}
\newtheorem{example}{Example}

\theoremstyle{remark}

\title{The true detection probability versus the subjective detection probability of a uniformly  optimal search plan}

\author{Liang Hong\footnote{Department of Mathematical Sciences,  The University of Texas at Dallas, 800 West Campbell Road, Richardson, TX 75080, USA. Tel.:~972-883-2161. Email address: liang.hong@utdallas.edu.}}
\date{\today}

\usepackage{tikz}
\usetikzlibrary{automata,arrows,positioning,calc}
\usepackage{listings}
\usepackage[labelformat=empty]{caption}
\usepackage{subfig}
\usepackage{caption}

\begin{document}

\maketitle

\begin{abstract}

This article investigates the difference between the true detection probability and the subjective probability of a uniformly optimal search plan.  Its main contributions are multi-fold.  First, it provides a set of examples to show that, in terms of the true detection probability, the uniformly optimal search plan may or may not be optimal.  Secondly,  it establishes that the true detection probability of the uniformly optimal search plan based on a composite prior can be less than that of the composite uniformly search plan based on different priors.  Next,  it argues that an open problem is unsolvable.  Finally, it shows that the true detection probability of the uniformly optimal search plan converges to one as the search time approaches infinity.

\smallskip

\emph{Keywords and phrases:} Bayesian learning; composite probability maps; optimal search theory; search and rescue; stationary targets.

\end{abstract}

\section{Introduction}

The theory of optimal search originated in the US Navy's Anti-Submarine Warfare Operations Research Group, which sought an efficient method for detecting hostile submarines during World War II (e.g., Koopman, 1946, 1956a, b, c).  Its later development was also mainly driven by the practical needs of maritime search missions (e.g., Stone and Stanshire 1971; Stone 1973, 1975, 1976; Richardson and Stone 1971; Richardson and Discenza 1980; Stone et al.  2014; Vermeulen and Brink 2017; Bourque 2019).  Most recently,  search games have gained significant attention in the literature; see, for instance,  Clarkson et al. (2020),  Lidbetter (2020),  Alpern et al. (2021),  and Lin (2021).

The uniformly optimal search plan is a cornerstone of the optimal search theory.  It has been successfully applied in practice (e.g.,  Richardson and Stone 1971; Richardson et al.  1980; Stone 1992; Stone et al.  2014).  The uniformly optimal search plan is the theoretical backbone of the U.S. Coast Guard's Search and Rescue Optimal Planning System (SAROPS) and its predecessor Computer Assisted Search Planning (CASP); see, for example,  Stone (1975) and Kratzke et al. (2010).  Arkin (1964) first established sufficient conditions for the existence of a uniformly optimal search plan in the Euclidean search space.   His work was generalized by Stone (1973, 1975, 1976).   The uniformly optimal search plan maximizes the subjective detection probability at each moment of search.  Its properties have been widely studied (e.g., Stone 1975; Stone et al. 2016; Hong 2024).  However,  the true detection probability of the uniformly optimal search plan has rarely been examined until recently (e.g., Hong 2025).  

This article aims to investigate the difference between the true detection probability and the subjective detection probability of the uniformly optimal search plan.  Its main contributions are as follows.  First,  it provides a set of examples to demonstrate some interesting facts regarding the relationship between the true detection probability and the subjective detection probability of the uniformly optimal search plan. In particular, it shows that
\begin{enumerate}
\item[(i)]the true detection probability of a uniformly optimal search may always equal its subjective detection probability;
\item[(ii)]the true detection probability of a uniformly optimal search may never equal its subjective detection probability;
\item[(iii)]the true detection probability of a uniformly optimal search plan might always be less than that of another search plan;
\end{enumerate}
These examples provide useful insight to both researchers and practitioners. For example,  (iii) implies that a uniformly optimal search plan does not necessarily maximize the true detection probability, although its definition guarantees it always maximizes the subjective detection probability.  It also implies that the true mean time to detection of a uniformly optimal search plan can be less than that of another search plan.  

Next, this article investigates the challenging situation when there are inconsistent target distributions.  In this case, analysts often generate a composite target distribution and obtain a uniformly optimal search plan based on it (e.g., Richardson and Discenza 1980; Stone 1992; Stone et al. 2014).   Intuitively, there is a reasonable alternative: first, obtain uniformly optimal search plans based on inconsistent target distributions; then,  create a composite search plan.  We will see that
\begin{enumerate}
\item[(iv)]the true detection probability of the uniformly optimal search plan based on a composite target distribution can be less than that of a composite search plan based on different target distributions.
\end{enumerate}

Given the above limitations of the uniformly optimal search plan, it is natural to ask whether we can find a search plan that maximizes the true detection probability at every moment of search.  This is an open problem proposed in Hong (2025).  We will see that this problem is unsolvable.  

Finally, this article establishes a new property of the uniformly optimal search plan: the true detection probability and the subjective detection probability of a uniformly optimal search plan both converge to one.  The practical interpretation of this result is as follows: if the search is expected to be protracted, then the optimal search plan will be approximately optimal even in terms of the true detection probability.

The remainder of the article is organized as follows.  Section~2 sets the stage by reviewing the search problem and the uniformly optimal search plan.  Section~3 proves Statements~(i)--(iii).  Section~4 establishes Statement~(iv). For each of these statements,  we provide both a discrete example and a continuous example.  Section~5 shows that the aforementioned open problem is unsolvable, and Section~6 shows that the limiting true detection probability of the uniformly optimal search plan is one.  Finally,  Section~7 concludes the article with some remarks.

\section{Notation and setup}
Consider the problem of searching for an unknown stationary target where the amount of available effort is limited.  Since the exact location $x$ of the target is unknown, we take a Bayesian approach to quantify the uncertainty of $x$ by specifying a non-degenerate \emph{target distribution} for it, whose (cumulative) distribution function and density function are denoted as $\Pi$ and $\pi$, respectively.  The \emph{possibility area}, denoted by $X$,  is the support of the target distribution.  We assume $X\subseteq \mathbb{R}^n$ for some positive integer $n$.   The \emph{discrete case} and the \emph{continuous case} refer to the cases where $X$ is countable and uncountable,  respectively.  In the discrete case, we always assume that $X$ is a subset of $\{1, 2, \ldots\}$ and the smallest region over which search effort can be placed is represented by a cell; in the continuous case, we assume that the search effort is allocated continuously.

Let $\mathbb{R}_+=[0, \infty)$. An \emph{allocation} on $X$ is a function $f:X\rightarrow \mathbb{R}_+$ in the discrete case.  In the continuous case,  an \emph{allocation} on $X$ is a function $f:X\rightarrow \mathbb{R}_+$ such that $\int_Af(x)dx$ equals the amount of search effort allocated in $A$ where $A\subseteq X$.   Clearly,  $\sum_{x\in X}f(x)$ and $\int_Xf(x)dx$ are the total effort in the discrete case and the continuous case, respectively. 

A \emph{detection function} $d: X\times \mathbb{R}_+\rightarrow [0, 1]$ accounts for the imperfection of the sensor.
In the discrete case, $d(x, y)$ denotes the conditional probability of detecting the target when $y$ amount of effort is allocated to cell $x$ given that the target is in cell $x$; in the continuous case,   $d(x, y)$ is the conditional probability of detecting the target if the effort density equals $y$ at $x$ given that the target is at $x$.  Throughout, we assume the detection function has been either correctly derived from physical laws or reliably estimated from repeated testing.  
Thus,  the detection function is objective and depends on the sensor used.  A detection function $d$ is \emph{regular} if $d(x, 0)=0$ and $\partial d(x, y)/\partial y$ is continuous, positive, and strictly decreasing for all $x\in X$.

Given a target distribution function $\Pi$,  an allocation $f$, and a detection function $d$,  $P[f]$ denotes the \emph{subjective probability of detection}:
\begin{equation}
\label{eq:subjectiveprob}
P[f]=\left\{
		                           \begin{array}{ll}
		                           \sum_{x\in X}d(x, f(x))\pi(x),& \hbox{in the discrete case,} \\
					\int_Xd(x, f(x))\pi(x)dx, & \hbox{in the continuous case.} 
		                          \end{array}
		                         \right.
\end{equation}
In contrast,  $P^{\#} [f]$ denotes the \emph{true/objectives probability of detection}: 
\begin{equation}
\label{eq:trueprob}
P^{\#}[f]=d(x_0, f(x_0)),
\end{equation}
where $x_0$ denotes the true cell that contains the target in the discrete case and the true target location in the continuous case. 

We define a cost function $c: X \times \mathbb{R}_+ \rightarrow \mathbb{R}_+$.  In the discrete case,  $c(x,  y)$ stands for the cost of applying $y$ effort in cell $x$; in the continuous case, $c(x, y)$ symbolizes the cost density of applying effort density $y$ at location $x$.  Hence,  if we let $C[f]$ denote the cost resulting from an allocation $f$,  then
\[
C[f]=\left\{
		                           \begin{array}{ll}
		                           \sum_{x\in X} c(x, f(x)),& \hbox{in the discrete case,} \\
					\int_Xc(x, f(x))dx, & \hbox{in the continuous case.} 
		                          \end{array}
		                         \right.
\]
Throughout, we assume $c(x, y)=y$ for all $x\in X$, that is,  the cost is proportional to allocation.

A \emph{search plan} on $X$ is a function $\varphi: X\times \mathbb{R}_+\rightarrow \mathbb{R}_+$ such that
\begin{enumerate}
\item[(i)]$\varphi(\cdot, t)$ is an allocation on $X$ for all $t\geq 0$;
\item[(ii)]$\varphi(x, \cdot)$ is an increasing function for all $x\in X$.
\end{enumerate}
Let $T$ be the time to find the target using a search plan $\varphi$, and let $\mu(\varphi)$ be the expectation of $T$ with respect to the subjective probability of detection.  Then 
\begin{equation*}
\mu(\varphi)=\int_0^\infty  (1-P[\varphi(\cdot, t)])dt.
\end{equation*}
Similarly,  let $\mu^{\#}$ be the expectation of $T$ with respect to the true probability of detection.  Then
\begin{equation}
\label{eq:truemean}
\mu^{\#} (\varphi)=\int_0^\infty(1-P^{\#}[\varphi(\cdot, t)])dt.
\end{equation}

The \emph{cumulative effort function} $E$ is a non-negative function with domain $\mathbb{R}_+$ such that $E(t)$ denotes the effort available by time $t$.  We assume $E$ is increasing and $E(t)>0$ for all $t>0$.  Given a target distribution function $\Pi$,  the \emph{uniformly optimal search plan for $\Pi$ and $E(t)$}  maximizes the subjective probability of detection at each moment of search, subject to the constraint $C[\varphi^\star(\cdot, t)]\leq E(t)$.  To be precise, let $\Phi(E)$ be the class of search plans $\varphi$ such that 
\begin{equation}
\label{eq:uniformdis1}
\sum_{x\in X}\varphi(x, t)=E(t), \quad \text{for all $t\geq0$},
\end{equation}
for the discrete case,  and 
\begin{equation}
\label{eq:uniformcon1}
\int_X\varphi(x, t)dx=E(t), \quad \text{for all $t\geq0$},
\end{equation}
for the continuous case. A search plan $\varphi^\star\in \Phi(E)$ is said to be \emph{uniformly optimal for $\Pi$ within $\Phi(E)$} if 
\begin{equation}
\label{eq:uniform}
P[\varphi^\star(\cdot, t)]=\max\{P[\varphi(\cdot, t)]\mid \varphi\in \Phi(E)\}\quad \text{for all $t\geq 0$},
\end{equation}
where $P[\varphi^\star]$ is the subjective probability of detection based on $\Pi$.  Example~2.2.9 of Stone (1975) shows that a uniformly optimal search plan does not always exist.  However, we have the following sufficient conditions for the existence of uniformly optimal search plans (e.g., Section~2.4 of Stone 1975).

\begin{theorem}
\label{thm:unifexistence}
\
\begin{enumerate}
\item[(i)]If $\Pi$ is a target distribution function for a discrete possibility area $X$, $d(x, 0)=0$,  and $d(x, \cdot)$ is continuous,  concave, and increasing for each $x\in X$,  there exists a uniformly optimal search plan for $\Pi$ within $\Phi(E)$.   
\item[(ii)]If $\Pi$ is a target distribution function for a continuous possibility area $X$,  $d(x, 0)=0$, and $d(x, \cdot)$ is increasing and right-continuous for each $x\in X$,  there exists a uniformly optimal search plan for $\Pi$  within $\Phi(E)$.   
\end{enumerate}
\end{theorem}

For a regular detection function,  we have a semi-closed form of the uniformly optimal search plan (e.g.,   Chapter~2 of Stone 1975; Chapter~5 of Washburn 2014).  

\begin{theorem}
\label{thm:unifopt}
If the cost function takes the form $c(x,y)=y$ for all $y\geq 0$ and $x\in X$ and the detection function is regular,  a uniformly optimal search plan $\varphi^\star$ within $\Phi(E)$ can be found for any target distribution function $\Pi$ as follows.  Define
\begin{equation}
q_x (y)=\pi(x)\frac{\partial }{\partial y}d(x, y), \quad x\in X \text{ and $y\geq 0$},
\end{equation}

\begin{equation}
q^{-1}_x(\lambda)=\left\{
		                           \begin{array}{ll}
		                           \text{the inverse function of $q_x(y)$  evaluated at $\lambda$},& \hbox{if $0<\lambda\leq q_x(0)$,} \\
					0, & \hbox{if $\lambda>q_x(0)$,} 
		                          \end{array}
		                         \right.
\end{equation}		                       
and 
\begin{equation}
Q(\lambda)= \left\{
		                           \begin{array}{ll}
		                          \sum_{x\in X} q_x^{-1}(\lambda),  & \hbox{in the discrete case,} \\
					 \int_X q_x^{-1}(\lambda)dx,  & \hbox{in the continuous case.} 
		                          \end{array}
		                         \right.
\end{equation}
Then a uniformly optimal search plan for $\Pi$ within  $\Phi(E)$ is given by $\varphi^\star(x, t)=q_x^{-1}(Q^{-1}(E(t)))$ where $Q^{-1}$ is the inverse function of $Q$. 
\end{theorem}

\section{$P[\varphi^\star(\cdot, t)]$ versus $P^{\#}[\varphi^\star(\cdot, t)]$}

\subsection{$P[\varphi^\star(\cdot, t)]$ may equal $P^{\#}[\varphi^\star(\cdot, t)]$ for all $t\geq 0$}

\subsubsection{Discrete case}

\begin{theorem}
\label{thm:1discrete}
Suppose the possibility area $X$ is discrete, $x_0\in X$,  and the target distribution is uniform on $X$.  If the detection function $d$ is regular and homogeneous, then $P[\varphi^\star(\cdot, t)]=P^{\#}[\varphi^\star(\cdot, t)]$ for all $t\geq 0$ and $\mu(\varphi^\star)=\mu^{\#}(\varphi^\star)$.
\end{theorem}

\begin{proof}
Without loss of generality, we assume that $X=\{1, 2, \ldots, m\}$.  By assumption, the target density function is given by 
\[
\pi(x)=\left\{
		                           \begin{array}{ll}
		                           \frac{1}{m},& \hbox{for $x\in X$,} \\
					0,  & \hbox{otherwise.} 
		                          \end{array}
		                         \right.
\]
Since $d$ is homogeneous,  we have $\frac{\partial}{\partial y}d(x,y)=d'(y)$ and 
\[
q_x(y)=\left\{
		                           \begin{array}{ll}
		                           \frac{1}{m} d'(y),& \hbox{if $x\in D$,} \\
					0,  & \hbox{otherwise,} 
		                          \end{array}
		                         \right.
\]
Therefore,
\begin{eqnarray*}
q_x^{-1}(\lambda) &=& \left\{
		                           \begin{array}{ll}
		                          \text{inverse of $d'(y)$ evaluated at $\frac{\lambda}{\pi(x)}$},& \hbox{for $0<\lambda \leq q_x(0)$ and $x\in D$,} \\
					0,  & \hbox{otherwise,} 
		                          \end{array}
		                         \right.\\
		                         &=& \left\{
		                           \begin{array}{ll}
		                          \text{inverse of $d'(y)$ evaluated at $\lambda m$},& \hbox{for $0<\lambda \leq \frac{d'(0)}{m}$ and $x\in D$,} \\
					0,  & \hbox{otherwise,} 
		                          \end{array}
		                         \right.
\end{eqnarray*}
and 
\[
Q(\lambda)=\sum_{x\in X} q_x^{-1}(\lambda)=m q_x^{-1}(\lambda).
\]
Therefore, 
\[
Q^{-1}(K)=\left(q_x^{-1}\right)^{-1}\left(K/m\right)=q_x \left(K/m \right),
\]
where the amount of effort $K\geq 0$.  It follows that
\begin{eqnarray*}
\varphi^\star(x, t)&=&
q_x^{-1}(Q^{-1}(E(t)))\\
&=&\left\{
		                           \begin{array}{ll}
		                           q_x^{-1}\left(q_x \left(E(t)/m \right)\right),& \hbox{for $0<q_x \left(E(t)/m\right) \leq q_x(0)$ and $x\in D$,} \\
					0,  & \hbox{otherwise,} 
		                          \end{array}
		                         \right.\\
&=&\left\{
		                           \begin{array}{ll}
		                         \frac{E(t)}{m}, & \hbox{$x\in D$,} \\
					0,  & \hbox{otherwise,} 
		                          \end{array}
		                         \right.		                         
\end{eqnarray*}
where the last equality follows from the fact that $0<q_x \left(E(t)/V_D\right) \leq q_x(0)$ always holds (because $q_x(y)=\pi(x)d'(y)$ is strictly decreasing in $y$ for each $x\in X$). Thus, 
\begin{eqnarray*}
P[\varphi^\star(\cdot, t)] &=& \sum_{x\in X} d(x, \varphi^\star(x, t))\pi(x)
                                         = \sum_{x \in X} d(\varphi^\star(x, t))\pi(x) =\sum_{x\in X} d(E(t)/m)\pi(x) \\
                                         &= &d(E(t)/m)=d(x_0, \varphi^\star(\cdot, t))=P^{\#}[\varphi^\star(\cdot, t)].
\end{eqnarray*}
Hence,  $\mu(\varphi^\star)=\mu^{\#}(\varphi^\star)$.
\end{proof}

\noindent \textbf{Remark.} Theorem~\ref{thm:1discrete} does not extend to the case where the detection function $d$ is non-homogeneous. To see this, let $m=2$, $x_0=1$, and the detection $d$ be 
\[
d(x, y)=\left\{
		                           \begin{array}{ll}
		                           1-e^{-y},& \hbox{if $x=1$,} \\
					1-e^{-2y},  & \hbox{if $x=2$,}\\
					0, & \hbox{otherwise.}
		                          \end{array}
		                         \right.\\
\]
Then we have 
\[
q_x^{-1}(\lambda)=\left\{
		                           \begin{array}{ll}
		                           -\ln (2\lambda),& \hbox{if $x=1$,} \\
					-\frac{1}{2}\ln (\lambda),  & \hbox{if $x=2$,}\\
					0, & \hbox{otherwise.}
		                          \end{array}
		                         \right.\\
\]
Therefore,  $Q(\lambda)=-\ln 2 -\frac{3}{2}\ln(\lambda)$ and $Q^{-1}(K)=e^{-\frac{2}{3}(K +\ln 2)}$.  This implies
\[
\varphi^{\star}(x, t)=\left\{
		                           \begin{array}{ll}
		                           \frac{2}{3}E(t)-\frac{1}{3}\ln 2,& \hbox{if $x=1$,} \\
					 \frac{1}{3}E(t)+\frac{1}{3}\ln 2,  & \hbox{if $x=2$,}\\
					0, & \hbox{otherwise.}
		                          \end{array}
		                         \right.\\
\]
Then 
\begin{eqnarray*}
P[\varphi^\star(\cdot, t)] &=& \frac{1}{2}\left(1-e^{-\left[\frac{2}{3}E(t)-\frac{1}{3}\ln 2\right]}\right)+\frac{1}{2}\left(1-e^{-2\left[\frac{1}{3}E(t)+\frac{1}{3}\ln 2\right]}\right) \\
&=& 1-2^{-2/3} e^{-\frac{2}{3}E(t)}-2^{-5/3} e^{-\frac{2}{3}E(t)}.
\end{eqnarray*}
But 
\[
P^{\#}[\varphi^\star(\cdot, t)]=1-e^{-\left[\frac{2}{3}E(t)-\frac{1}{3} \ln 2\right]}=1-2^{-2/3}e^{-\frac{2}{3}E(t)}>P[\varphi^\star(\cdot, t)]
\quad \text{for all $\geq 0$}. 
\]
It also follows that $\mu(\varphi^\star)< \mu^{\#}(\varphi^\star)$.

\begin{example}[Discrete case]
\label{ex:discrete1}
Suppose the possibility area $X=\{1, 2\}$, the true target location is cell $1$, and the search budget at time $t$ is $E(t)$.   Let the target density function be $\pi(1)=\pi(2)=1/2$.  Assume the detection function $d$  is  $d(x, y)=1-e^{-y}, \ y\geq 0$ for all $x\in X$.  Theorem~\ref{thm:unifopt} implies that the uniformly optimal search plan $\varphi^\star$ exists and is given by 
\[
\varphi^\star (1, t)=\varphi^\star (2, t)=\left\{
		                           \begin{array}{ll}
		                            0,  & \hbox{if $t=0$,} \\
					 \frac{E(t)}{2},  & \hbox{if $t>0$.} 
		                          \end{array}
		                         \right.
		                         \]
Thus, the subjective probability of detection equals
\begin{eqnarray*}
P[\varphi^\star(\cdot, t)] &=& \pi(1) d(1,  \varphi^\star (1, E(t)))+\pi(2)d(2, \varphi^\star (2, E(t)))\\
			&=& p \left(1-e^{- E(t)/2} \right)+(1-p)\left( 1-e^{-E(t)/2}\right) \\
			&=& 1-e^{-E(t)/2}, \quad \text{for all $t>0$}.
\end{eqnarray*}

Since the true target location is cell $1$,  the true probability of detection equals
\begin{eqnarray*}
P^{\#}[\varphi^\star(\cdot, t)] (p)&=&  d(1,  \varphi^\star (1, E(t)))\\
				&=& 1-e^{- E(t)/2}\quad \text{for all $t>0$}.			
\end{eqnarray*}
Also, we have $P[\varphi^\star(\cdot, 0)]=P^{\#}[\varphi^\star(\cdot, 0)]=0$,  because the detection probability (true or subjective) of any search plan (uniformly optimal or not) is $0$ for $t=0$.  Therefore,  $P[\varphi^\star(\cdot, t)]=P^{\#}[\varphi^\star(\cdot, t)]$ for all $t\geq 0$.  This also implies $\mu(\varphi^\star)=\mu^{\#}(\varphi^\star)$.
\end{example}

\subsubsection{Continuous case}

\begin{theorem}
\label{thm:1continuous}
Suppose $X$ is continuous and bounded, $x_0\in X$, and the target distribution is uniform on $X$.  If the detection function $d$ is regular and homogeneous, then $P[\varphi^\star(\cdot, t)]=P^{\#}[\varphi^\star(\cdot, t)]$ for all $t\geq 0$ and $\mu(\varphi^\star)=\mu^{\#}(\varphi^\star)$.
\end{theorem}

\begin{proof}
By assumption, the target density function is given by 
\[
\pi(x)=\left\{
		                           \begin{array}{ll}
		                           \frac{1}{V_X},& \hbox{for $x\in X$,} \\
					0,  & \hbox{otherwise,} 
		                          \end{array}
		                         \right.
\]
where $V_X$ is the volume of $X$, i.e., $V_X=\int_X 1dx$.  Since $d$ is homogeneous,  we have $\frac{\partial}{\partial y}d(x,y)=d'(y)$ and 
\[
q_x(y)=\left\{
		                           \begin{array}{ll}
		                           \frac{1}{V_X} d'(y),& \hbox{if $x\in X$,} \\
					0,  & \hbox{otherwise,} 
		                          \end{array}
		                         \right.
\]
Therefore,
\begin{eqnarray*}
q_x^{-1}(\lambda) &=& \left\{
		                           \begin{array}{ll}
		                          \text{inverse of $d'(y)$ evaluated at $\frac{\lambda}{\pi(x)}$},& \hbox{for $0<\lambda \leq q_x(0)$ and $x\in X$,} \\
					0,  & \hbox{otherwise,} 
		                          \end{array}
		                         \right.\\
		                         &=& \left\{
		                           \begin{array}{ll}
		                          \text{inverse of $d'(y)$ evaluated at $\lambda V_X$},& \hbox{for $0<\lambda \leq \frac{d'(0)}{V_X}$ and $x\in X$,} \\
					0,  & \hbox{otherwise,} 
		                          \end{array}
		                         \right.
\end{eqnarray*}
and 
\[
Q(\lambda)=\int_X q_x^{-1}(\lambda)dx=V_X\times q_x^{-1}(\lambda).
\]
Therefore, 
\[
Q^{-1}(K)=\left(q_x^{-1}\right)^{-1}\left(K/V_X\right)=q_x \left(K/V_X\right),\ K\geq 0,
\]
and 
\begin{eqnarray*}
\varphi^\star(x, t)&=&
q_x^{-1}(Q^{-1}(E(t)))\\
&=&\left\{
		                           \begin{array}{ll}
		                           q_x^{-1}\left(q_x \left(E(t)/V_X\right)\right),& \hbox{for $0<q_x \left(E(t)/V_X\right) \leq q_x(0)$ and $x\in X$,} \\
					0,  & \hbox{otherwise,} 
		                          \end{array}
		                         \right.\\
&=&\left\{
		                           \begin{array}{ll}
		                         \frac{E(t)}{V_X}, & \hbox{$x\in X$,} \\
					0,  & \hbox{otherwise,} 
		                          \end{array}
		                         \right.		                         
\end{eqnarray*}
where the last equality follows from the fact that $q_x(y)=\pi(x)d'(y)$ is strictly decreasing in $y$ for each $x\in X$.  Therefore, 
\begin{eqnarray*}
P[\varphi^\star(\cdot, t)] &=& \int_X d(x, \varphi^\star(x, t))\pi(x) dx 
                                         = \int_X d(\varphi^\star(x, t))\pi(x) dx=\int_X d(E(t)/V_X)\pi(x) dx\\
                                         &= &d(E(t)/V_X)=d(x_0, \varphi^\star(\cdot, t))=P^{\#}[\varphi^\star(\cdot, t)].
\end{eqnarray*}
Hence, $\mu(\varphi^\star)=\mu^{\#}(\varphi^\star)$.
\end{proof}

\noindent \textbf{Remark.} This theorem does not extend to the case where $d$ is non-homogeneous.  To see this, let $X=(a, b)$ where  $0<a<x_0<b$ and $\ln(b/a)\neq (b-a)/x_0$,  let $d(x, y)=1-e^{-xy}, x, y>0$,  and let $\Pi$ be the uniform distribution on $(a, b)$.
Then 
\[
q_x(y) = \left\{
		                           \begin{array}{ll}
		                         \frac{x}{b-a}e^{-xy},& \hbox{if $a<x<b$,} \\
					0,  & \hbox{otherwise,} 
		                          \end{array}
		                         \right.
\]
It follows that
\[
q_x^{-1}(\lambda) = \left\{
		                           \begin{array}{ll}
		                         -\frac{1}{x} \ln\left[\frac{\lambda (b-a)} {x}\right],& \hbox{for $0<\lambda \leq 1/(b-a)$ and $x\in (a, b)$,} \\
					0,  & \hbox{otherwise,} 
		                          \end{array}
		                         \right.
\]
and 
\[
\varphi^\star(x, t)
=\left\{
		                           \begin{array}{ll}
		                           -\frac{1}{x} \ln\left[\frac{Q^{-1}(E(t)) (b-a)} {x}\right], & \hbox{for $0<Q^{-1}(E(t)) \leq 1/(b-a)$ and $x\in (a, b)$,} \\
					0,  & \hbox{otherwise.} 
		                          \end{array}
		                         \right.
\]
Therefore, 
\begin{eqnarray*}
P[\varphi(\cdot, t)]=\int_Xd(x, \varphi^\star(x, t))\pi(x)dx=1-(\ln b-\ln a)Q^{-1}(E(t)),
\end{eqnarray*}
and
\begin{eqnarray*}
P^{\#}[\varphi(\cdot, t)]=d(x_0, \varphi^\star(x, t))=1-\frac{(b-a)Q^{-1}(E(t))}{x_0}.
\end{eqnarray*}
It follows that $P[\varphi(\cdot, t)]\neq P^{\#}[\varphi(\cdot, t)]$ if and only if $\ln (b/a)\neq \frac{b-a}{x_0}$.  


\begin{example}
\label{ex:continuous1}
Let $x_0$ be the origin $(0, 0)$ and let $X=B_2(r)\subseteq \mathbb{R}^2$, where $B_2(r)$ denotes the disc centered at $(0, 0)$ with a radius $r>0$.  Suppose  the detection function is $d(x, y)=1-e^{-y}, \ y\geq 0$ for all $x\in X$.  Consider a  target distribution as follows:
\[
\pi(x)=\left\{
		                           \begin{array}{ll}
		                           \frac{1}{\pi r^2},& \hbox{for $x\in B_2(r)$,} \\
					0,  & \hbox{otherwise.} 
		                          \end{array}
		                         \right.
\]
Then $\frac{\partial}{\partial y}d(x,y)=d'(y)=e^{-y}$ and 
\[
q_x(y)=\left\{
		                           \begin{array}{ll}
		                           \frac{e^{-y}}{\pi r^2} ,& \hbox{if $x\in B_2(r)$,} \\
					0,  & \hbox{otherwise,} 
		                          \end{array}
		                         \right.
\]
Therefore,
\begin{eqnarray*}
q_x^{-1}(\lambda) 		            &=& \left\{
		                           \begin{array}{ll}
		                          \text{$-\ln (\lambda \pi r^2)$},& \hbox{for $0<\lambda \leq \frac{1}{\pi r^2}$ and $x\in B_2(r)$,} \\
					0,  & \hbox{otherwise,} 
		                          \end{array}
		                         \right.
\end{eqnarray*}
and 
\[
Q(\lambda)=\int_X q_x^{-1}(\lambda)dx= \left\{
		                           \begin{array}{ll}
		                          \text{$-\pi r^2 \ln (\lambda \pi r^2)$},& \hbox{for $0<\lambda \leq \frac{1}{\pi r^2}$ and $x\in B_2(r)$,} \\
					0,  & \hbox{otherwise.} 
		                          \end{array}
		                         \right.
\]
It follows that
\[
Q^{-1}(K)=\left\{
		                           \begin{array}{ll}
		                           \frac{e^{-\frac{K}{\pi r^2}}}{\pi r^2} ,& \hbox{if $x\in B_2(r)$,} \\
					0,  & \hbox{otherwise,} 
		                          \end{array}
		                         \right.
\]
and 
\begin{eqnarray*}
\varphi^\star(x, t)&=&
q_x^{-1}(Q^{-1}(E(t)))\\
&=&\left\{
		                           \begin{array}{ll}
		                         \frac{E(t)}{\pi r^2}, & \hbox{$x\in B_2(r)$,} \\
					0,  & \hbox{otherwise.} 
		                          \end{array}
		                         \right.		                         
\end{eqnarray*}
That is,  the uniformly optimal search plan distributes available effort uniformly on the support of the target distribution.  The subjective detection probability is 
\begin{eqnarray*}
P[\varphi^\star(\cdot, t)] &=& \int_X d(x, \varphi^\star(x, t)) \pi(x)dx=\int_{B_2(r)} \left(1-e^{-\frac{E(t)}{\pi r^2}}\right)\frac{1}{\pi r^2}dx=1-e^{-\frac{E(t)}{\pi r^2}}, \quad \text{for all $t>0$}.
\end{eqnarray*}
The true detection probability equals
\[
P^{\#}[\varphi^\star(\cdot, t)]=d(x_0, \varphi^\star(x_0, t))=1-e^{-\frac{E(t)}{\pi r^2}}, \quad \text{for all $t>0$}.
\]
Since $P[\varphi^\star(\cdot, 0)]=P^{\#}[\varphi^\star(\cdot, 0)]=0$ always holds,  we have $P[\varphi^\star(\cdot, t)]=P^{\#}[\varphi^\star(\cdot, t)]$ for all $t\geq 0$.  Therefore,  $\mu(\varphi^\star)=\mu^{\#}(\varphi^\star)$.
\end{example}

\subsection{$P[\varphi^\star(\cdot, t)]$ may never equal $P^{\#}[\varphi^\star(\cdot, t)]$ for all $t> 0$}

\medskip

\begin{example}[Discrete case]
\label{ex:discrete2}
Consider Example~1 in Hong (2025), which slightly generalizes the example on Page~3 of Stone (1975).  Specifically, let $X=\{1, 2\}$, $x_0=1$, and $E(t)>\ln [p/(1-p)]$ for all $t>0$.   We take the target distribution to be $\pi(1)=p$ and $\pi(2)=1-p$, where $1/2<p<1$.  The detection function $d$  is $d(x, y)=1-e^{-y}, \ y\geq 0$ for all $x\in X$.  By Theorem~\ref{thm:unifopt},  the uniformly optimal search plan $\varphi^\star$ exists and is given by
\[\varphi^\star (1, t)=\left\{
		                           \begin{array}{ll}
		                            E(t),  & \hbox{if $0<E(t)\leq  \ln \left(\frac{p}{1-p}\right) $,} \\
					 \frac{1}{2}\left[E(t)+\ln \left(\frac{p}{1-p}\right) \right],  & \hbox{if $E(t)>\ln \left(\frac{p}{1-p}\right) $,} 
		                          \end{array}
		                         \right.
		                         \]
and 
\[\varphi^\star (2, t)=\left\{
		                           \begin{array}{ll}
		                          0,  & \hbox{if $0< E(t) \leq \ln \left(\frac{p}{1-p}\right) $,} \\
					 \frac{1}{2}\left[E(t)-\ln \left(\frac{p}{1-p}\right) \right],  & \hbox{if $E(t)>\ln \left(\frac{p}{1-p}\right)$.} 
		                          \end{array}
		                         \right.
		                         \]
		                         Then the subjective probability of detection is
\begin{eqnarray*}
P[\varphi^\star(\cdot, t)] &=& \pi(1) d(1,  \varphi^\star (1, E(t)))+\pi(2)d(2, \varphi^\star (2, E(t)))\\
			&=& p\left(1-e^{- \frac{1}{2}\left[E(t)+\ln \left(\frac{p}{1-p}\right) \right]}\right)+(1-p)\left(1-e^{- \frac{1}{2}\left[E(t)-\ln \left(\frac{p}{1-p}\right) \right]}\right) \\
			&=& 1-2\sqrt{p(1-p)}e^{-E(t)/2}.
\end{eqnarray*}

Since $x_0=1$ and $E(t)>\ln [p/(1-p)]$ for all $t>0$,  the true probability of detection is 
\begin{eqnarray*}
P^{\#}[\varphi^\star(\cdot, t)]&=&  d(1,  \varphi^\star (1, E(t)))\\
				&=& 1-e^{- \frac{1}{2}\left[E(t)+\ln \left(\frac{p}{1-p}\right) \right]}\\
				&=& 1-e^{-E(t)/2}\sqrt{\frac{1}{p}-1}.				
\end{eqnarray*}
It is straightforward to verify that $P[\varphi^\star(\cdot, t)] \geq P^{\#}[\varphi^\star(\cdot, t)] $ if and only if $p\geq 1/2$, and that the equality holds if and only if $p=1/2$.   Since $p>1/2$, we have $P[\varphi^\star(\cdot, t)]\neq P^{\#}[\varphi^\star(\cdot, t)]$ for all $t>0$.  
\end{example}


\begin{example}[Continuous case]
\label{ex:continuous2}
Consider the setup in Examples~2.2.1 and 2.2.7 of Stone (1975).  Suppose the search is conducted at a constant speed $v$ using a sensor with sweeping width $W$.  Hence,  $E(t)=Wvt$.  Assume the possibility area is $X=\mathbb{R}^2$ and the true target location $x_0$ is $(0, 0)$.  The target distribution is bivariate normal with the following density function:
\begin{equation}
\pi(x_1, x_2)=\frac{1}{2\pi\sigma^2}e^{-\frac{x_1^2+x_2^2}{2\sigma^2}}, \quad (x_1, x_2)\in X= \mathbb{R}^2,
\end{equation}
where $\sigma>0$, and the detection function is $d(x, y)=1-e^{-y}$ for all $x\in X$ and $y\geq 0$. For convenience, we will use polar coordinates.  Example~2.2.1 of Stone (1975) shows that  the uniformly optimal plan exists and is given by
\begin{eqnarray*}
\varphi^\star((r, \theta), t) &=& \left\{
		                           \begin{array}{ll}
		                           \left(\frac{ E(t)}{\pi \sigma^2}\right)^{1/2}-\frac{r^2}{2\sigma^2},& \hbox{if $r^2\leq 2\sigma^2\left(\frac{E(t)}{\pi \sigma^2}\right)^{1/2}$,} \\
					0, & \hbox{if $r^2> 2\sigma^2\left(\frac{E(t)}{\pi \sigma^2}\right)^{1/2}$,} 
		                          \end{array}
		                         \right.\\
		                         &=& 
		                          \left\{
		                           \begin{array}{ll}
		                          \left[H\sqrt{t}-\frac{r^2}{2\sigma^2}\right],& \hbox{if $0\leq r\leq R_x(t)$,} \\
					0, & \hbox{if $r>R_x(t)$,} 
		                          \end{array}
		                         \right.
\end{eqnarray*}		                        	                         
where $R^2(t)=2\sigma^2H\sqrt{t}$ and $H=\sqrt{Wv/\pi\sigma^2}$.   By Example~2.2.7 of Stone (1975), the subjective probability of detection is 
\begin{equation}
P[\varphi^\star(\cdot, t)]=1-(1+H\sqrt{t})e^{-H\sqrt{t}}, \quad \text{for all $t>0$}.
\end{equation}
Since $x_0=(0,0)$ can be written as $(0, \theta)$ for some angle $\theta$ in polar coordinates,  the true probability of detection  equals 
\begin{eqnarray*}
P^{\#}[\varphi^\star(\cdot, t)] &=& 
		                             1-e^{-H \sqrt{t}}, \quad \text{forall $t>0$}.
\end{eqnarray*}
It follows that $P[\varphi^\star(\cdot, t)]\neq P^{\#}[\varphi^\star(\cdot, t)]$ for all $t> 0$.
\end{example}

\subsection{There might exist a search plan $\varphi\neq \varphi^\star$ such that $P[\varphi^\star(\cdot, t)]>P[\varphi(\cdot, t)]$ but  $P^{\#}[\varphi^\star(\cdot, t)]<P^{\#}[\varphi(\cdot, t)]$ for all $t>0$}

\medskip

\begin{example}[Discrete case]
\label{ex:discrete3}
Take the same setup as in Example~\ref{ex:discrete2} with $p=2/3$ but assume $E(t)> \ln 4$ for all $t>0$; let $\varphi^\star$ be the corresponding uniformly optimal search plan.  Consider another search plan $\varphi$:
\[\varphi (1, t)=\left\{
		                           \begin{array}{ll}
		                            E(t)/2  & \hbox{if $0<E(t)\leq  \ln 4 $,} \\
					 \frac{1}{2}\left[E(t)+\ln 4 \right],  & \hbox{if $E(t)>\ln 4 $,} 
		                          \end{array}
		                         \right.
		                         \]
and 
\[\varphi(2, t)=\left\{
		                           \begin{array}{ll}
		                          E(t)/2,  & \hbox{if $0< E(t) \leq \ln 4 $,} \\
					 \frac{1}{2}\left[E(t)-\ln 4 \right],  & \hbox{if $E(t)>\ln4$.} 
		                          \end{array}
		                         \right.
		                         \]
Then $\varphi\neq \varphi^\star$ and $\varphi\in \Phi(E)$. By the definition of the uniformly optimal search plan,  we have $P[\varphi^\star(\cdot, t)]>P[\varphi(\cdot, t)]$ for all $t>0$.  In fact,  it is straightforward to verify that
\[
P[\varphi^\star(\cdot, t)] = 1-\frac{2\sqrt{2}}{3}e^{-E(t)/2}>1-e^{-E(t)/2}=P[\varphi(\cdot, t)].
\]
However, 
\begin{eqnarray*}
P^{\#}[\varphi^\star(\cdot, t)] &=& d(1, \varphi^\star(1, E(t)))=1-\frac{\sqrt{2}}{2}e^{-E(t)/2}\\
                                        & <& 1-\frac{1}{2}e^{-E(t)/2}=d(1, \varphi(1, E(t)))
                                        = P^{\#}[\varphi(\cdot, t)].
\end{eqnarray*}
for all $t> 0$.

\begin{figure}[h]
\begin{center}
\subfigure[Subjective detection probabilities]{\scalebox{0.45}{\includegraphics{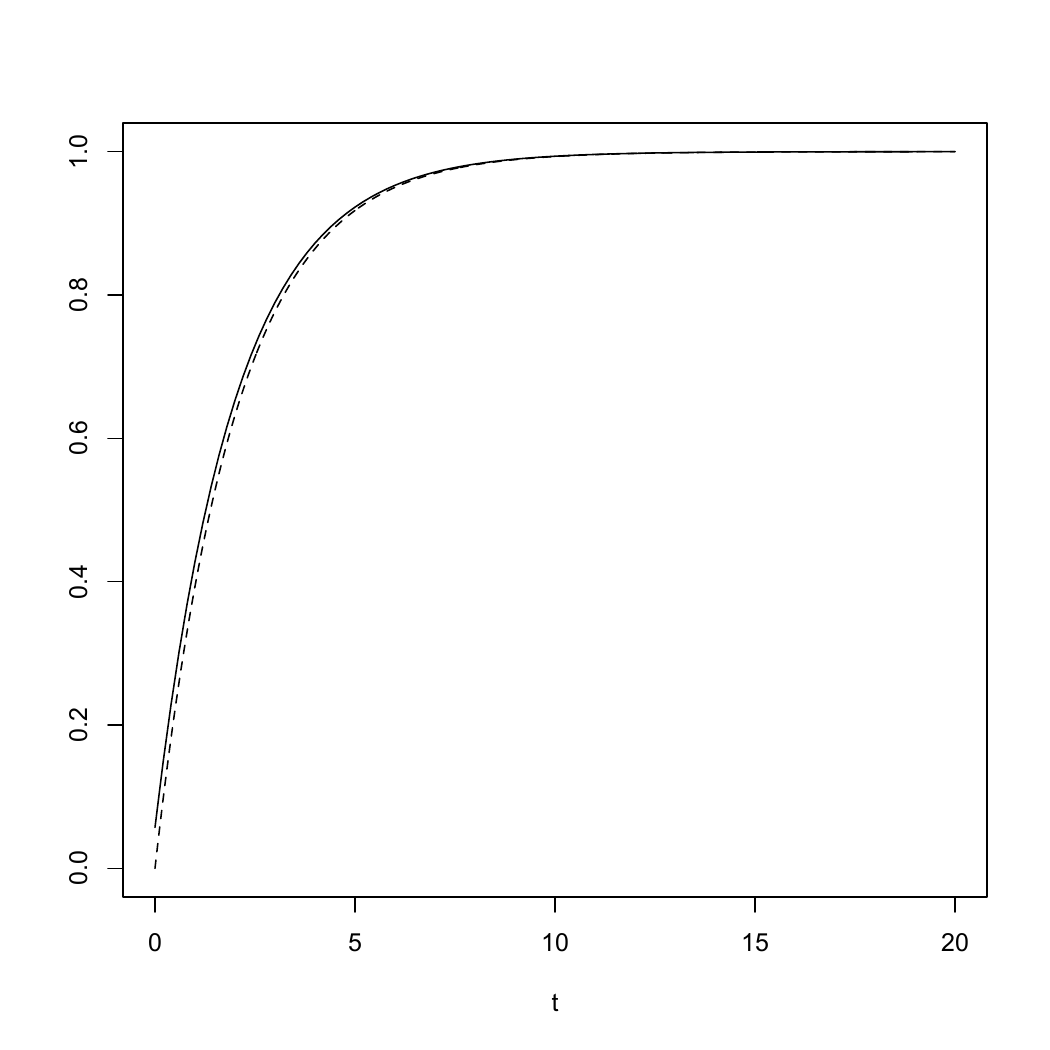}}}
\subfigure[True detection probabilities]{\scalebox{0.45}{\includegraphics{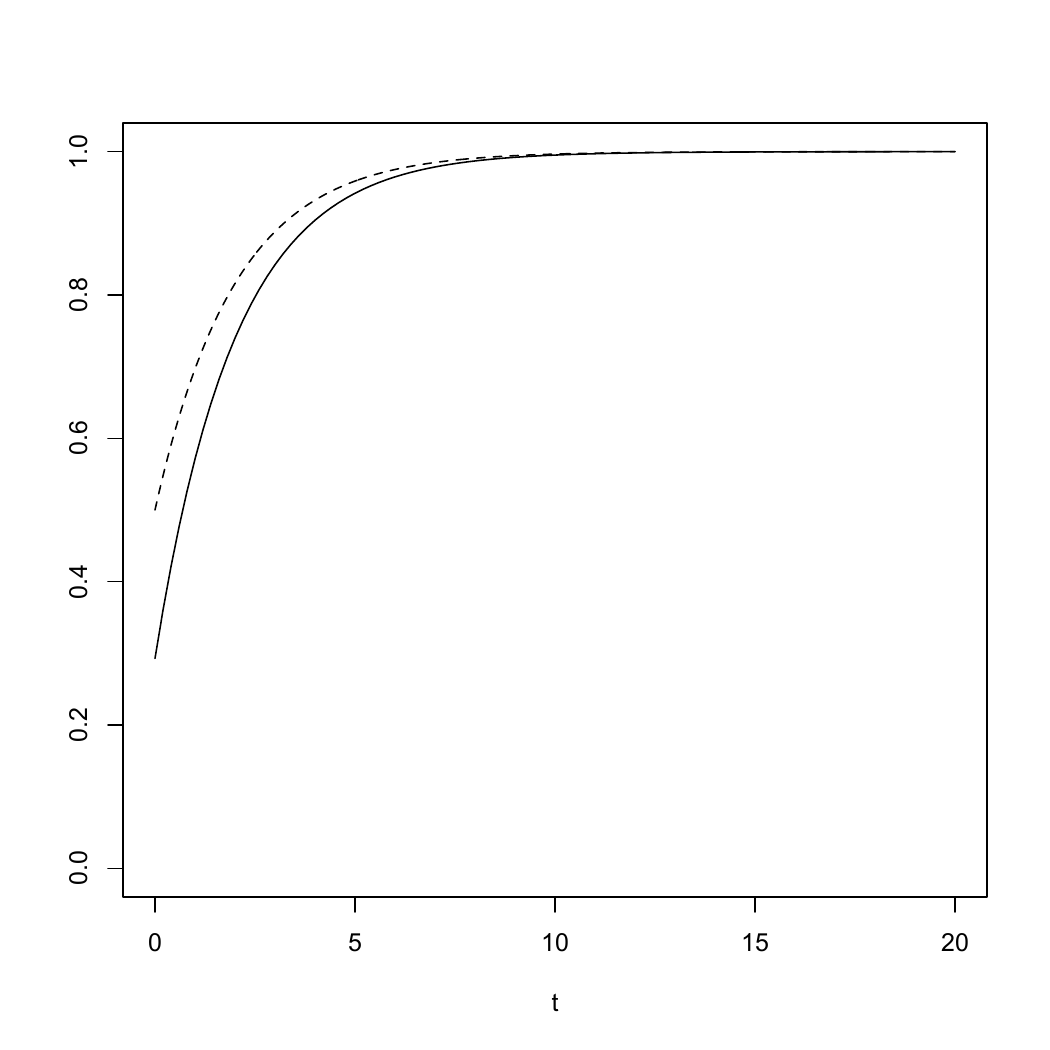}}}
\caption{Figure~1: The true and subjective detection probabilities of $\varphi^\star$ and $\varphi$ as a function of $t$ in Example~\ref{ex:discrete3}, when $p=2/3$ and $E(t)> \ln 4$ for all $t>0$. In Panels~(a) and (b), the solid line and the dashed line represent the detection probability of $\varphi^\star$ and $\varphi$, respectively.}
\label{fig:ex-discrete3}
\end{center}
\end{figure}

Panel (a) of Figure~\ref{fig:ex-discrete3} plots $P[\varphi^\star(\cdot, t)]$ (solid line) and $P[\varphi(\cdot, t)]$ (dashed line) as functions of $t$, when $p=2/3$ and $E(t)> \ln 4$; Panel~(b) does the same for $P^{\#}[\varphi^\star(\cdot, t)]$ (solid line) and $P^{\#}[\varphi(\cdot, t)]$ (dashed line).  As we can see,  while the difference between $P[\varphi^\star(\cdot, t)]$ and $P[\varphi(\cdot, t)]$ is negligible,  the same cannot be said for the difference between $P^{\#}[\varphi^\star(\cdot, t)]$ and $P^{\#}[\varphi(\cdot, t)]$. 
\end{example}

\bigskip
\begin{example}[Continuous case]

\label{ex:continuous3}
Take the same setup as in Example~\ref{ex:continuous2} with $\sigma=2$ and $W=v=1$. Then the target density function is
\[
\pi(x_1, x_2)=\frac{1}{8\pi}e^{-\frac{x_1^2+x_2^2}{8}}, \quad (x_1, x_2)\in \mathbb{R}^2,
\]
 and the uniformly optimal search plan $\varphi^\star$ is given by 
\begin{eqnarray*}
\varphi^\star((r, \theta), t) &=& \left\{
		                           \begin{array}{ll}
		                           \sqrt{t/4 \pi }-\frac{r^2}{8},& \hbox{if $r^2\leq 8\sqrt{t/4\pi}$,} \\
					0, & \hbox{if $r^2> 8\sqrt{t/4\pi}$.} 
		                          \end{array}
		                         \right.
\end{eqnarray*}
Also, 
\begin{eqnarray*}
P[\varphi^\star(\cdot, t)] &=& 1-(1+\sqrt{t/4\pi})e^{-\sqrt{t/4\pi}}, \\
P^{\#}[\varphi^\star (\cdot, t)] &=& 1-e^{-\sqrt{t/4\pi}}.
\end{eqnarray*}
Now consider another search plan:
\begin{eqnarray*}
\varphi (((x_1, x_2), t) &=&  \left\{
		                           \begin{array}{ll}
		                           \sqrt{t/ \pi }-\frac{r^2}{2},& \hbox{if $r^2\leq 2\sqrt{t/\pi}$,} \\
					0, & \hbox{if $r^2> 2\sqrt{t/\pi}$.} 
		                          \end{array}
		                         \right.
\end{eqnarray*}
Then 
\[
P^{\#}[\varphi (\cdot, t)]=1-e^{-\sqrt{t/\pi}}>1-e^{-\sqrt{t/4\pi}}=P^{\#}[\varphi^\star (\cdot, t)]. 
\]
Since $\varphi^\star$ is the uniformly optimal search plan, we have $P[\varphi^\star(\cdot, t)]>P[\varphi(\cdot, t)]$ for all $t> 0$.  Indeed, we have 
\begin{eqnarray*}
P[\varphi^\star(\cdot, t)] &=& \int_0^{2\pi}\int_0^\infty (1-e^{-\varphi((r, \theta), t)}) \frac{1}{8\pi}e^{-\frac{r^2}{8}}rdrd\theta \\
				    &=& 1-\frac{1}{4}\int_0^{\left(2\sqrt{t/\pi}\right)^{1/2}} e^{-\frac{r^2}{8}} e^{-\sqrt{t/ \pi }+\frac{r^2}{8}}rdr-\frac{1}{4}\int_{\left(2\sqrt{t/\pi}\right)^{1/2}} ^\infty e^{-\frac{r^2}{8}}rdr  \\
				    &=& 1-\frac{4}{3}e^{-\sqrt{t/\pi}/4}+\frac{1}{3}e^{-\sqrt{t/\pi}}.
\end{eqnarray*}

\begin{figure}[h]
\begin{center}
\subfigure[Subjective detection probabilities]{\scalebox{0.45}{\includegraphics{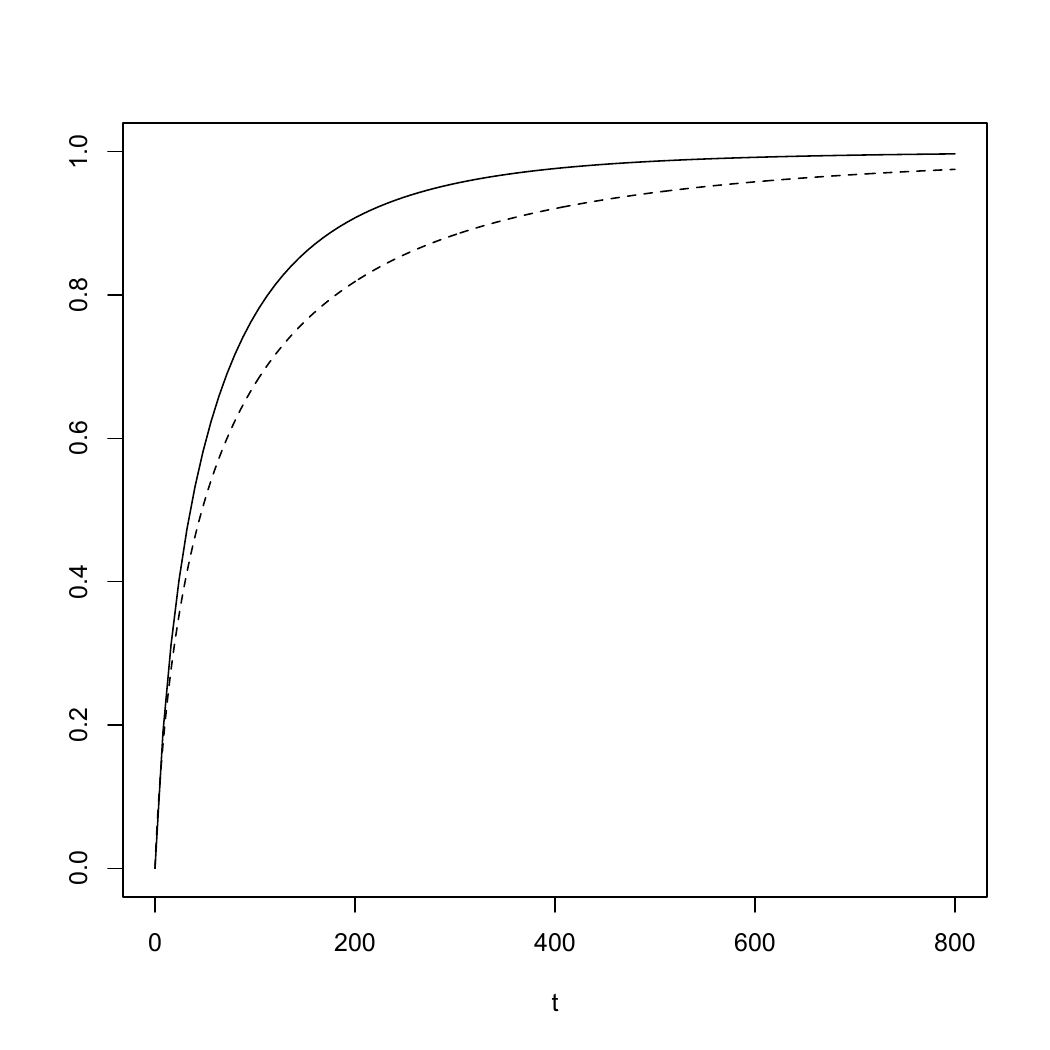}}}
\subfigure[True detection probabilities]{\scalebox{0.45}{\includegraphics{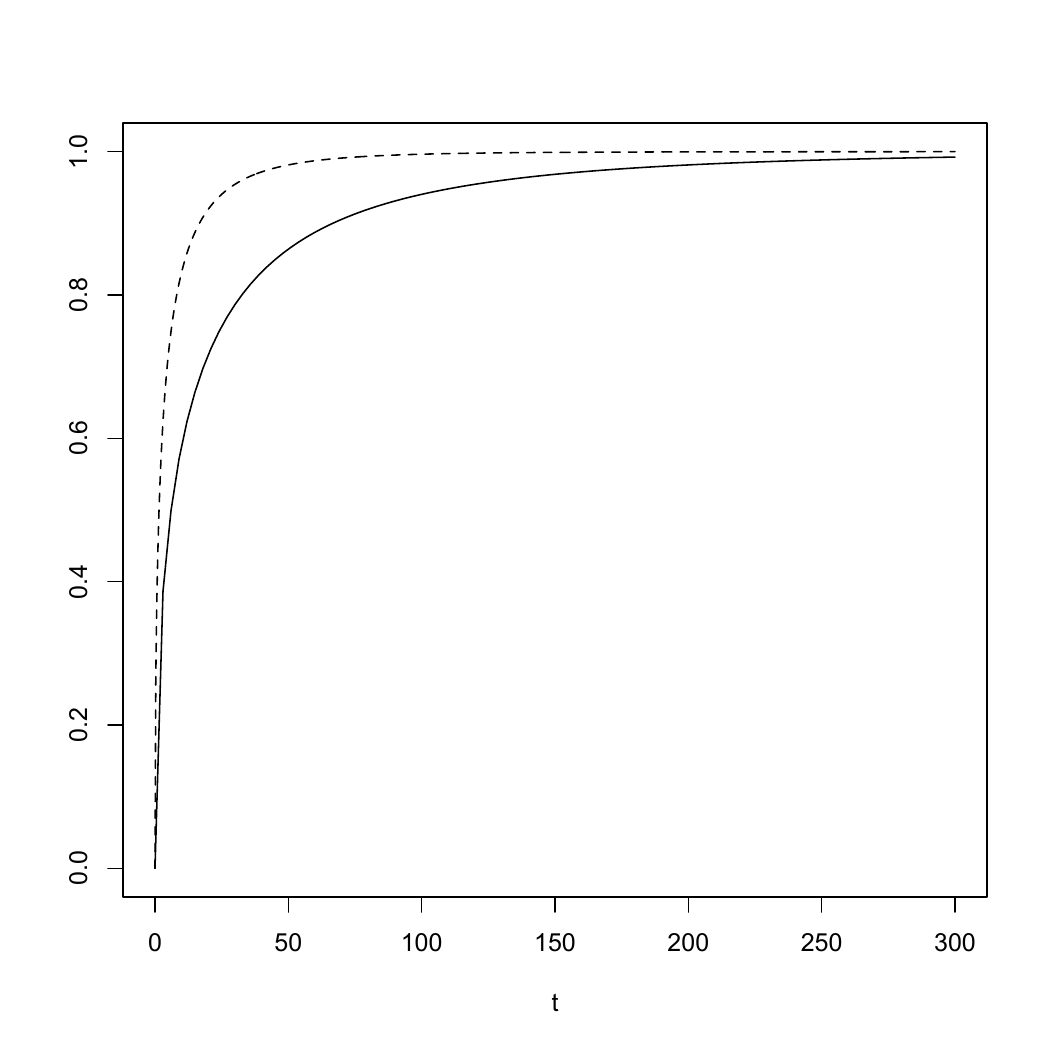}}}
\caption{Figure~2: The true and subjective detection probabilities of $\varphi^\star$ and $\varphi$ as functions of $t$ in Example~\ref{ex:continuous3}, when $\sigma=2$ and $W=v=1$.  In each panel, the solid line and the dashed line represent the detection probability of $\varphi^\star$ and $\varphi$, respectively.}
\label{fig:ex-continuous3}
\end{center}
\end{figure}

Figure~\ref{fig:ex-continuous3} provides plots  of $P[\varphi^\star(\cdot, t)]$, $P[\varphi(\cdot, t)]$, $P^{\#}[\varphi^\star(\cdot, t)]$, and $P^{\#}[\varphi^\star(\cdot, t)]$ as functions of $t$.  In both Panels~(a) and (b), the detection probabilities of $\varphi^\star$ and $\varphi$ are represented by the solid line and the dashed line, respectively.  We see that the difference between $P[\varphi^\star(\cdot, t)]$ and $P[\varphi(\cdot, t)]$ is significant, and the same can be said for the difference between $P^{\#}[\varphi^\star(\cdot, t)]$ and $P ^{\#}[\varphi(\cdot, t)]$.
\end{example}

\section{Inconsistent target distributions}

In practice, analysts often face inconsistent target distributions due to conflicting information.  This is a challenging situation because subjective detection probabilities (of different uniformly optimal search plans) based on different target distributions are not comparable.  A common approach to circumvent this challenge is as follows: first,  generate a composite target distribution based on inconsistent information (A composite target distribution is often called a hierarchical prior in statistics; see, for example,   Section~3.6  of Berger 2006.); then, obtain a uniformly optimal search plan based on the composite target distribution.  This approach has proved to be successful in several high-profile searches (e.g., Stone 1992 and Stone et al.  2014).  Apart from this approach, there is an intuitively reasonable alternative: first,  obtain a uniformly optimal search plan based on each of these inconsistent target distributions; then, create a composite search plan based on these different uniformly optimal search plans.  By definition of the uniformly optimal search plan, the first approach always outperforms the second in terms of the subjective detection probability.  However,  the next two examples show that the second approach can outstrip the first in terms of the true detection probability.

\begin{example}[Discrete case]
\label{ex:discrete4}
Suppose $X=\{1, 2\}$, $x_0=1$, $d(x, y)=1-e^{-cy}, y\geq 0$ for all $x\in X$ where $c>0$,  and there are two inconsistent target density functions $\pi_1$ and $\pi_2$: 
\begin{enumerate}
\item[]$\pi_1(1)=p_1$ and $\pi_1(2)=1-p_1$,
\item[]$\pi_2(1)=p_2$ and $\pi_2(2)=1-p_2$,
\end{enumerate}
where $0<p_2<1/2<p_1<1$.  The weights assigned to $\pi_1$  and $\pi_2$ are $w$ and $1-w$ respectively, where $1/2<w<1$.  It is assumed that $E(t)>\max\{\ln(p/(1-p)), \ln(p_1/(1-p_1)), \ln(p_2/(1-p_2))\}$, where $p=wp_1+(1-w)p_2$. 

The density function of the composite target distribution is
\[
\pi(1)=p \text{ and $\pi(2)=1-p$}.
\]
By almost the same argument as in Example~\ref{ex:discrete2}, we know  
\[
P^{\#}[\varphi^\star(\cdot, t)]=1-e^{-cE(t)/2}\left(\frac{1}{p}-1\right)^{c/2},
\]
where $\varphi^\star$ is the uniformly optimal search plan based on the composite target distribution.

For $i=1, 2$, let $\varphi_i^\star$ be the uniformly optimal search plan based on $\pi_i$.  Put 
\[
\varphi^\star_c=w\varphi^\star_1 +(1-w)\varphi^\star_2.
\]
Clearly, $\varphi_1^\star$ and $\varphi_2^\star$ both belong to $\Phi_{\Pi}(E)$.  Hence $\varphi^\star_c\in \Phi_{\Pi}(E)$.  By Example~\ref{ex:discrete2} and symmetry, we have 
\begin{eqnarray*}
\varphi_c^\star(1, t) &=& \frac{w}{2}\left[E(t)+\ln\left(\frac{p_1}{1-p_1}\right)\right]+\frac{1-w}{2}\left[E(t)-\ln\left(\frac{1-p_2}{p_2}\right)\right] \\
&=& \frac{1}{2}\left[E(t)+w\ln\left(\frac{p_1}{1-p_2}\right)-(1-w)\ln\left(\frac{1-p_2}{p_2}\right)\right].
\end{eqnarray*}
It follows that
\begin{eqnarray*}
P^{\#}[\varphi_c^\star(\cdot, t)] 
&=& 1-e^{-cE(t)/2}\left(\frac{1-p_1}{p_1}\right)^{cw/2}\left(\frac{1-p_2}{p_2}\right)^{c(1-w)/2}.
\end{eqnarray*}

Now take $c=0.3$, $p_1=0.99$, $p_2=0.17$, and $w=0.75$. Then $p=0.75$, $\max\{\ln(p/(1-p)), \ln(p_1/(1-p_1)), \ln(p_2/(1-p_2))\}=4.59512$, and
\[
P^{\#}[\varphi^\star(\cdot, t)]=1-0.82345e^{-0.3 E(t)} < 1-0.63287 e^{-0.3 E(t)}=P^{\#}[\varphi_c^\star(\cdot, t)]\quad \text{for all $t>0$.}
\]

\begin{figure}[h]
\begin{center}
\subfigure[True detection probabilities]{\scalebox{0.45}{\includegraphics{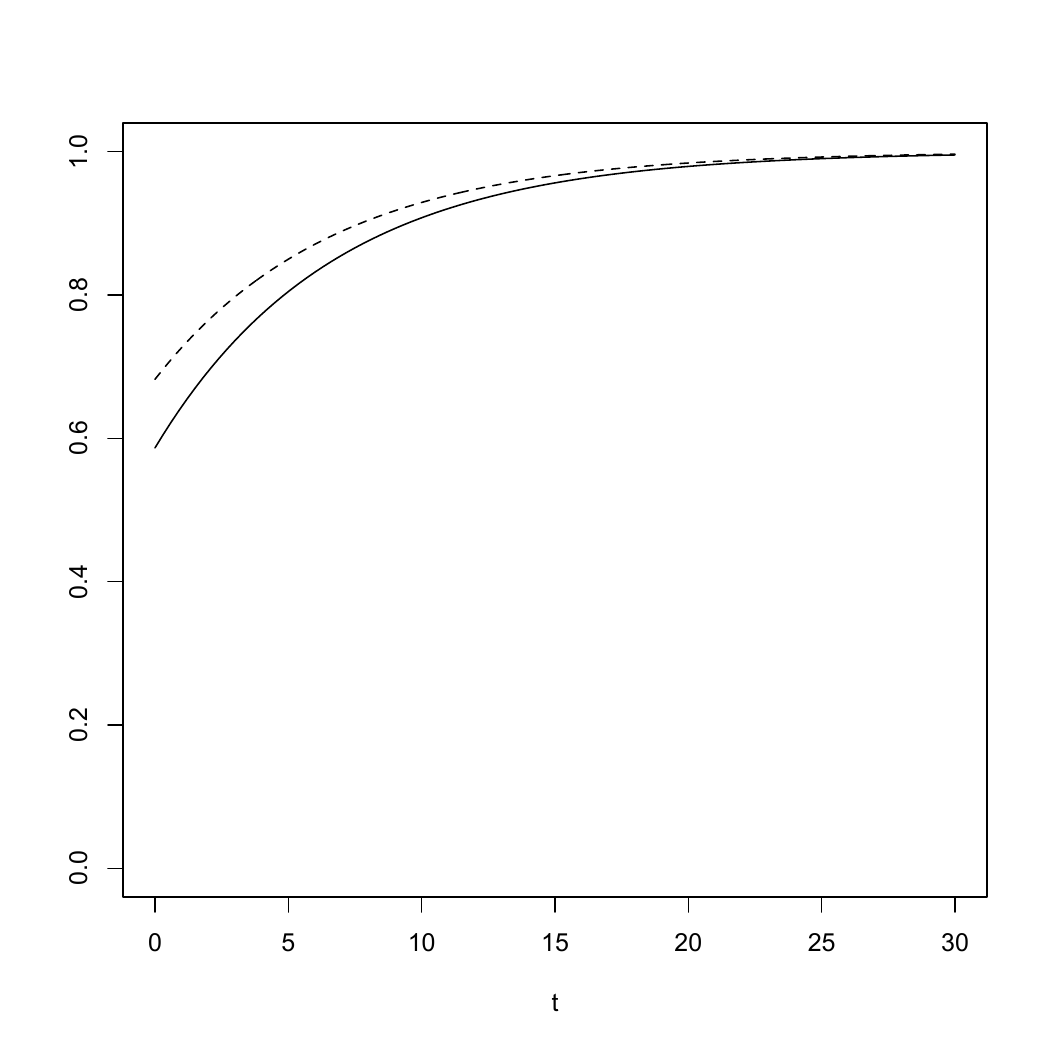}}}
\subfigure[Difference of the true detection probabilities]{\scalebox{0.45}{\includegraphics{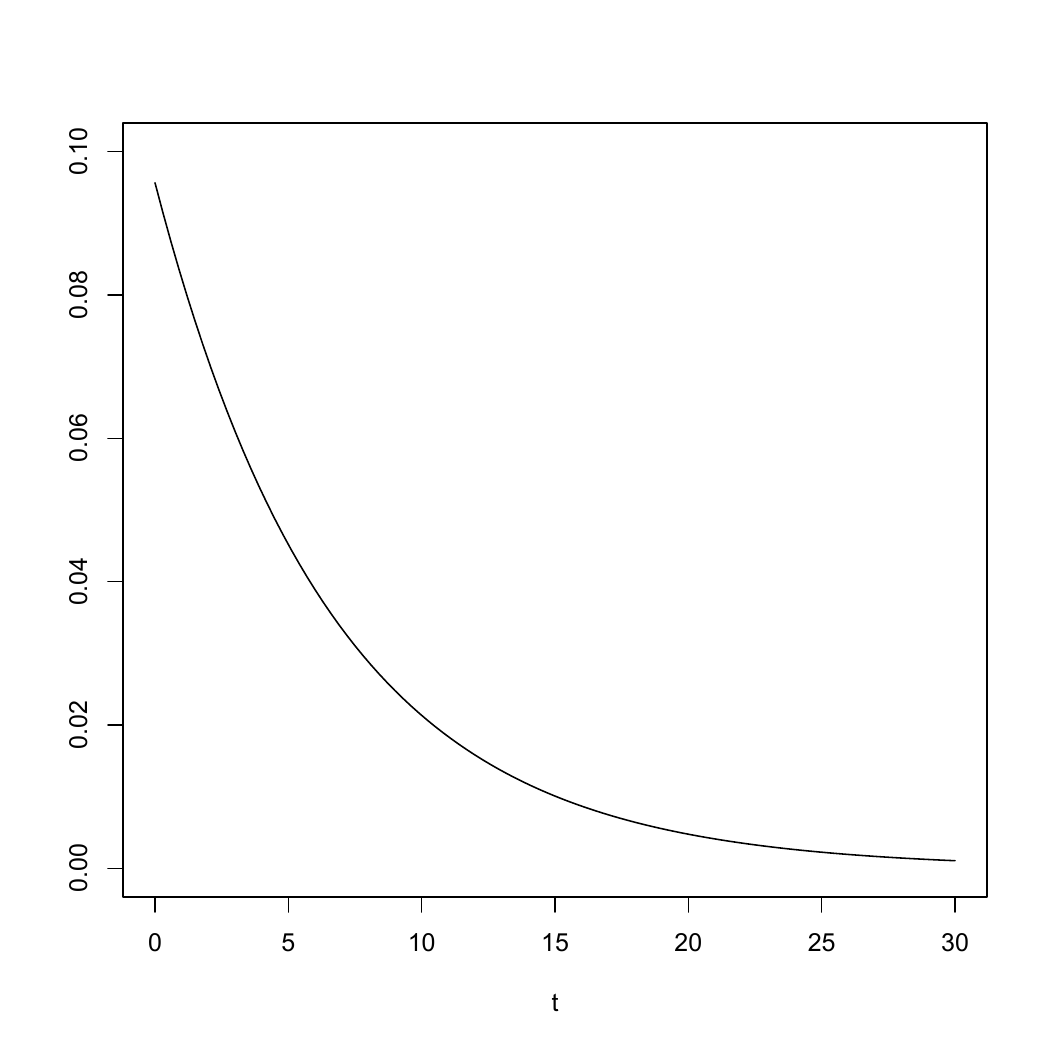}}}
\caption{Figure~3: Comparison of the true detection probabilities of $\varphi^\star$ and $\varphi^\star_c$ in Example~\ref{ex:discrete4}, when $E(t)=4.59512+t$, $c=0.3$, $p_1=0.99$, $p_2=0.17$, and $w=0.75$.   In Panel~(a), the solid line and the dashed line represent $P^{\#}[\varphi^\star(\cdot, t)]$ and $P^{\#}[\varphi_c^\star(\cdot, t)]$, respectively.  Panel~(b) plots $P^{\#}[\varphi_c^\star(\cdot, t)]-P^{\#}[\varphi^\star(\cdot, t)]$. }
\label{fig:ex-discrete4}
\end{center}
\end{figure}

Figure~\ref{fig:ex-discrete4} provides plots of $P^{\#}[\varphi^\star(\cdot, t)]$ and  $P^{\#}[\varphi^\star_c(\cdot, t)]$ over the time interval $[0, 30]$ when $E(t)=4.59512+t$, $c=0.3$, $p_1=0.99$, $p_2=0.17$, and $w=0.75$.  The difference between $P^{\#}[\varphi^\star(\cdot, t)]$ and $P^{\#}[\varphi^\star_c(\cdot, t)]$ is not negligible at the beginning, but it becomes smaller as $t$ increases.

\end{example}


\begin{example}[Continuous case]
\label{ex:continuous4}
Consider the same setup as in Example~\ref{ex:continuous2} except that there are two inconsistent target density functions $\pi_1$ and $\pi_2$: 
\begin{eqnarray*}
\pi_1(x_1, x_2) &=& \frac{1}{2\pi\sigma^2_1}e^{-\frac{x_1^2+x_2^2}{2\sigma_1^2}}, \quad (x_1, x_2)\in X= \mathbb{R}^2, \\
\pi_1(x_1, x_2) &=& \frac{1}{2\pi\sigma^2_2}e^{-\frac{x_1^2+x_2^2}{2\sigma_2^2}}, \quad (x_1, x_2)\in X= \mathbb{R}^2,
\end{eqnarray*}
where $\sigma_1, \sigma_2>0$. To create a composite target distribution, we attach weights $w$ and $1-w$ to $\pi_1$ and $\pi_2$ respectively.  By the closure property of the normal distribution,  the density function of the composite target distribution equals
\[
\pi (x_1, x_2) =\frac{1}{2\pi\sigma^2}e^{-\frac{x_1^2+x_2^2}{2\sigma^2}}, \quad (x_1, x_2)\in X= \mathbb{R}^2,
\]
where $\sigma^2=w\sigma_1^2+(1-w)\sigma_2^2$.  We know from Example~\ref{ex:continuous2} that
\[
P^{\#}[\varphi^\star(\cdot, t)]=1-e^{-\sqrt{\frac{E(t)}{\pi \sigma^2}}},
\]
where $\varphi^\star$ is the uniformly optimal search plan based on $\pi$.

Let $\varphi_i^\star$ denote the  uniformly optimal search plan based on  $\pi_i$ for $i=1, 2$.  Then the composite search plan is 
\[
\varphi^\star_c=w\varphi^\star_1 +(1-w)\varphi^\star_2.
\]
Since $x_0=(0, 0)$, we know from Example~\ref{ex:continuous2} that 
\[
\varphi^\star_c(x_0, t)=w \sqrt{\frac{E(t)}{\pi \sigma_1^2}}+(1-w) \sqrt{\frac{E(t)}{\pi \sigma_2^2}}=\left(\frac{w}{\sigma_1}+\frac{1-w}{\sigma_2}\right)\sqrt{\frac{E(t)}{\pi}}.
\]
Hence,  
\[
P^{\#}[\varphi^\star_c(\cdot, t)]=1-e^{-\left(\frac{w}{\sigma_1}+\frac{1-w}{\sigma_2}\right)\sqrt{\frac{E(t)}{\pi}}}.
\]
Now take $\sigma_1=2$, $\sigma_2=0.5$, and $w=0.5$. Then $\sigma=1.4577$ and 
\[
P^{\#}[\varphi^\star(\cdot, t)]=1-e^{-0.686\sqrt{\frac{E(t)}{\pi}}}<1-e^{-1.25\sqrt{\frac{E(t)}{\pi}}}=P^{\#}[\varphi^\star_c(\cdot, t)] \quad \text{for all $t>0$.}
\]

\begin{figure}[h]
\begin{center}
\subfigure[True detection probabilities]{\scalebox{0.45}{\includegraphics{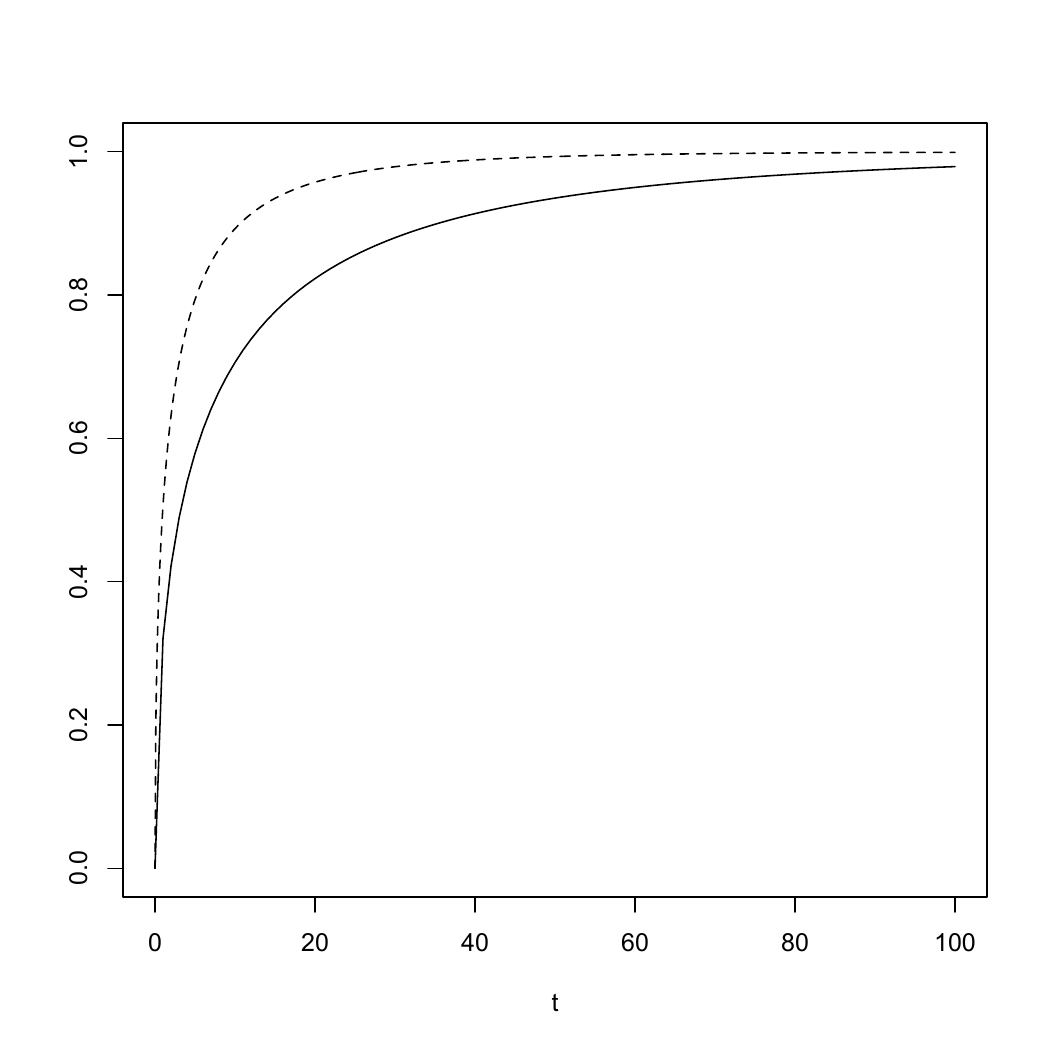}}}
\subfigure[Difference of the true detection probabilities]{\scalebox{0.45}{\includegraphics{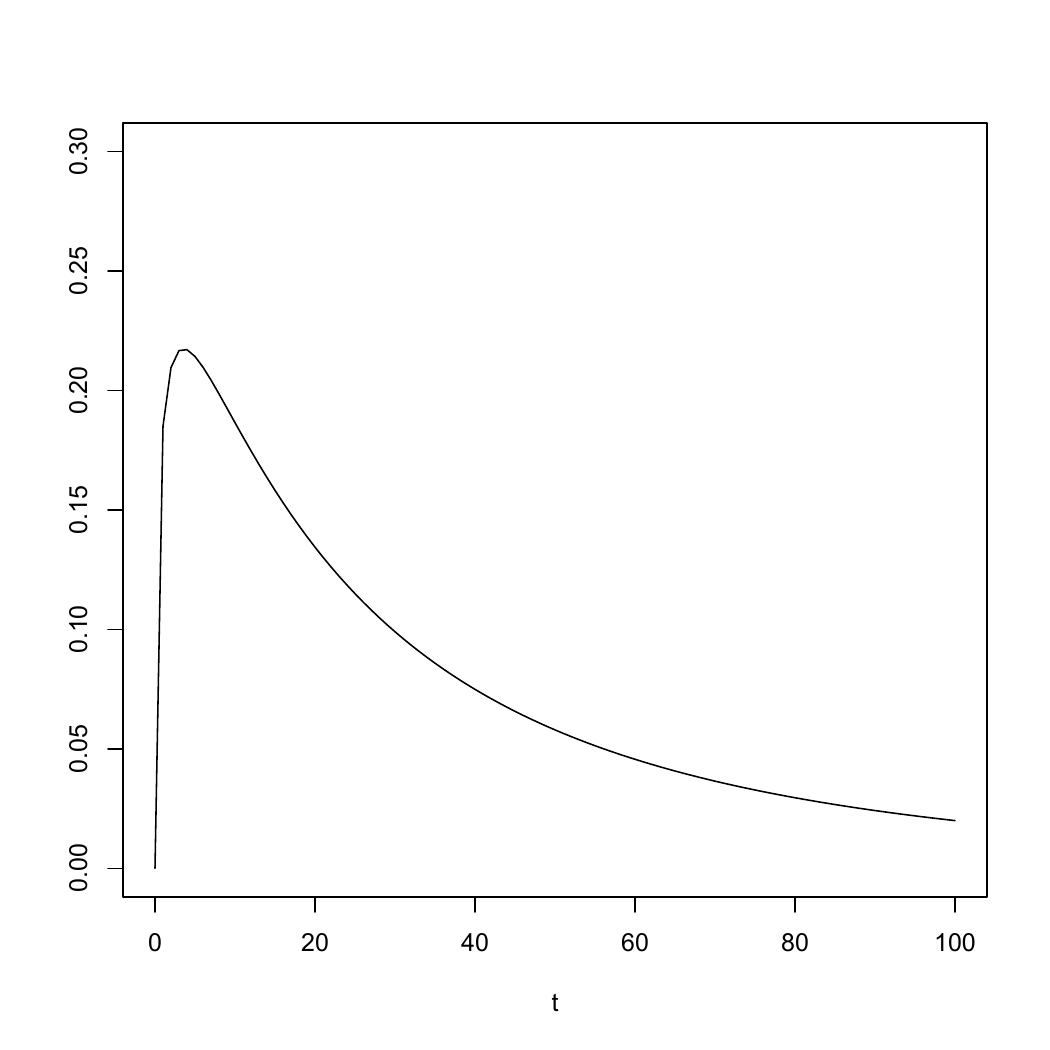}}}
\caption{Figure~4: Comparison of the true detection probabilities of $\varphi^\star$ and $\varphi^\star_c$ in Example~\ref{ex:continuous4}, when $E(t)=t$, $\sigma_1=2$, $\sigma_2=0.5$, and $w=0.5$.   In Panel~(a), the solid line and the dashed line represent $P^{\#}[\varphi^\star(\cdot, t)]$ and $P^{\#}[\varphi_c^\star(\cdot, t)]$ respectively.  Panel~(b) plots $P^{\#}[\varphi_c^\star(\cdot, t)]-P^{\#}[\varphi^\star(\cdot, t)]$. }
\label{fig:ex-continuous4}
\end{center}
\end{figure}

Figure~\ref{fig:ex-continuous4} compares $P^{\#}[\varphi^\star(\cdot, t)]$ and  $P^{\#}[\varphi^\star_c(\cdot, t)]$ when  $E(t)=t$, $\sigma_1=2$, $\sigma_2=0.5$, and $w=0.5$.   Clearly, the difference between $P^{\#}[\varphi^\star(\cdot, t)]$ and $P^{\#}[\varphi^\star_c(\cdot, t)]$ is significant. Also, the difference between $P^{\#}[\varphi^\star(\cdot, t)]$ and $P^{\#}[\varphi^\star_c(\cdot, t)]$  first increases, then decrease, and eventually approaches $0$.

\end{example}

\section{Unsolvability of an open problem}
The previous two sections demonstrate some limitations of the uniformly optimal search plan.  To overcome these limitations,  we need to find a search plan that maximizes the true detection probability at each moment of search. That is, to find a search plan $\psi^\star$ such that
\begin{equation}
\label{eq:optimal-true}
P^{\#}[\psi^\star(\cdot, t)]=\max_{\varphi\in \Phi(E)}P^{\#}[\varphi(\cdot, t)] \quad \text{for all $t\geq 0$.}
\end{equation}
This problem was first proposed by Hong (2025).  A solution to this problem would be an analyst's wildest dream. Unfortunately,  this problem is unsolvable.  To see this,  consider the following search plan 
\[
\psi(x, t)=\left\{
		                           \begin{array}{ll}
		                           E(t),& \hbox{if $x=x_0$,} \\
					0,  & \hbox{otherwise.} 
		                          \end{array}
		                         \right.
\]
Then
\[
P^{\#}[\psi(\cdot, t)]=\max_{\varphi\in \Phi(E)}P^{\#}[\varphi(\cdot, t)] \quad \text{for all $t\geq 0$.}
\]
Conversely, it is easy to see that any solution $\psi^\star$ to (\ref{eq:optimal-true}) must equal $\psi$.  Since $\psi$ puts all available effort $E(t)$ on the true target location at every moment $t>0$,  its construction requires us to know the true target location.  But we are uncertain about the true target location in the search problem.  Therefore,  this problem is unsolvable. 

What if we impose an additional restriction: the search plan must not put all effort on a single location? If we could obtain an objective distribution for the target location by repeating the experiment of how the target is lost infinitely often.  Then we could derive a search plan by using the formula for the uniformly optimal search plan, with the target distribution being replaced by this objective distribution. This search plan would be the solution to  (\ref{eq:optimal-true}).  However,  in any realistic case, such as a lost submarine or aircraft,  the same incident will not recur under identical conditions.  Hence,  this problem remains unsolvable.

\section{A limiting property of the uniformly optimal search plan}
Though the true and subjective detection probability of a uniformly optimal search plan can differ significantly,  we will see that they both converge to $1$ as $t$ goes to $\infty$ as far as the target distribution is not completely wrong.  The practical implication of this result is that the uniformly optimal search plan will be approximately optimal in terms of the true detection probability for a prolonged search.  In fact, Hong (2024) establishes that if $\varphi^\star$ is a uniformly optimal search plan, then $\lim_{t\rightarrow \infty} P[\varphi^\star(\cdot, t)]=1$.  The next theorem shows that the same can be said for the true detection probability $P^{\#}[\varphi^\star(\cdot, t)]$.  

\begin{theorem}
If the detection function $d$ is regular and $\lim_{y\rightarrow \infty} d(x, y)=1$ for all $x\in X$, the target distribution is continuous with a probability density function $\pi$,  $x_0$ belongs to the support of the target distribution, and $\lim_{t\rightarrow \infty}E(t)=\infty$, then $\lim_{t\rightarrow \infty} P^{\#}[\varphi^\star(\cdot, t)]=1$. 
\end{theorem}

\begin{proof}
By Theorem~\ref{thm:unifexistence}, the uniformly optimal search plan $\varphi^\star$ exists in this case.   Since $x_0$ belongs to the support of the target distribution, the true detection probability is 
\[
P^{\#}[\varphi^\star(\cdot, t)]=d(x_0, \varphi^\star(x_0, t))=d(x_0, q_{x_0}^{-1}(Q^{-1}(E(t)))).
\]
Since $Q$ is continuous on $(0, \infty)$ and strictly decreasing on the interval $(0, \sup\{\lambda\mid Q(\lambda)>0\})$, $Q^{-1}$ is continuous an strictly decreasing on $(0, q_x(0)]$.  Also,  we know $q_x^{-1}$ is decreasing on $(0, \infty)$ and $\lim_{\lambda \rightarrow 0} q_x^{-1}(\lambda)=\infty$ for all $x\in X$ such that $\pi(x)>0$ (e.g.,  Page~47 of Stone 1975). Thus, $\lim_{K\rightarrow \infty}q_x^{-1}(Q^{-1}(K))=\infty$.  It follows from the assumption $\lim_{t\rightarrow \infty}E(t)=\infty$ that 
\[
\lim_{t \rightarrow \infty}P^{\#}[\varphi^\star(\cdot, t)]=d\left(x_0, q_{x_0}^{-1}\left(Q^{-1}\left(\lim_{t \rightarrow \infty}E(t)\right)\right)\right)=1.
\]
\end{proof}
\noindent \textbf{Remarks.} First,  it is easy to see that this theorem also holds for the discrete case by the same argument. Secondly,  the assumption $\lim_{y\rightarrow \infty} d(x, y)=1$ for all $x\in X$ is not automatically satisfied by a regular detection function. To see this, consider $d(x, y)=c(1-e^{-y})$ for all $x\in X$ and $y\geq 0$ where $0<c<1$.  Finally,  the true detection probability of an arbitrary search plan in $\Phi(E)$ does not necessarily go to  $1$ as $t\rightarrow \infty$. To see this, consider the case where $X=\mathbb{R}^2$,  $x_0=(0, 0)$, the detection function is $d(x, y)=1-e^{-y}$ for all $x\in X$ and $t\geq 0$,  and a search plan $\varphi$ given by (in polar coordinates)
\[
\varphi((r, \theta), t)=\left\{
		                           \begin{array}{ll}
		                           e^{-t},& \hbox{for $0<r\leq R(t)$,} \\
					\frac{[Wvt-I(t)]}{\pi[\widetilde{R}^2(t)-R^2(t)]},  & \hbox{$R(t)<r\leq \widetilde{R}(t)$,} \\
					0,  & \hbox{otherwise,} 
		                          \end{array}
		                         \right.
\]
where the search is conducted at speed $v$ using sensor with a sweep width $W$, and $R^2(t)=2\sigma^2H\sqrt{t}$, $H=\sqrt{Wv/\pi \sigma^2}$,  $\widetilde{R}(t)=\sqrt{R^2(t)+(Wvt-I(t))/\pi}$,  and $I(t)=\int_0^{2\pi}\int_0^{R(t)} e^{-t}rdrd\theta$. Then we have
\[
\lim_{t\rightarrow \infty} P^{\#}[\varphi(\cdot, t)]=\lim_{t\rightarrow \infty} \left(1-e^{-e^{-t}}\right)=0.
\]

\section{Concluding remarks}

We established that the uniformly optimal search plan is not optimal in terms of the true detection probability. This fact should not be taken as a drawback of the uniformly optimal search plan.  It simply reveals the challenging nature of the search problem.  When practitioners face inconsistent target distributions, they can obtain different uniformly optimal search plans based on these conflicting target distributions.  We have shown that, in terms of the true detection probability, the uniformly optimal search plan based on the composite target distribution may be inferior to the composite search plan based on different uniformly optimal search plans.  These two facts prompted us to seek a solution to the following problem: to find a search plan that maximizes the true detection probability at each moment of search.  Unfortunately,  this problem is unsolvable according to our investigation.  Moreover,  we established that the true detection probability of the uniformly optimal search plan approaches one when the search time goes to infinity.

\section*{Acknowledgments}
I thank Jeffrey R. Cares,  Matthew Cosner, and Michael W. Kopp for a fruitful discussion during the 92nd MORS Symposium at Naval Postgraduate School. This discussion inspired several parts of this article.

\section*{Conflict of interest}
The author has no conflict of interest to declare.


\section*{References}
\begin{description}

\item{} Alpern, S., Chleboun, P., Katslkas, S.~and Lin, K.Y.~(2021).  Adversarial patrolling in a uniform. \emph{Operations Research}~70(1), 129-140.

\item{} Arkin, V.I.~(1964). Uniformly optimal strategies in search problems. \emph{Theory of Probability and its Applications}~9(4),  647--677. 

\item{} Berger, J.O.~(2006). \emph{Statistical Decision Theory and Bayesian Analysis}, Second Edition, Springer: New York.

\item{}Bourque, FA.~(2019).  Solving the moving target search problem using indistinguishable searchers. \emph{European Journal of Operational Research}~275, 45--52.

\item{} Clarkson, J., Glazebrook, K.D.~and Lin, K.Y.~(2020). Fast or slow: search in discrete locations with two search models. \emph{Operations Research}~68(2), 552--571.


\item{} Hong, L.~(2024).  On several properties of uniformly optimal search plans. \emph{Military Operations Research}~29(2), 95---106. 

\item{} Hong, L.~(2025).   On the true detection probability of the uniformly optimal search plan.  \emph{Journal of the Operational Research Society}, to appear, \url{https://arxiv.org/abs/2311.18226}.


\item{} Koopman, B.O.~(1946). Search an screening. \emph{Operations Evaluation Group Report No. 56 (unclassified).} Center for Naval Analysis, Rosslyn, Virginia. 

\item{} Koopman, B.O.~(1956a). The theory of search, I. Kinematic bases. \emph{Operations Research}~4, 324--346.

\item{} Koopman, B.O.~(1956b). The theory of search, II. Target detection. \emph{Operations Research}~4, 503--531.

\item{} Koopman, B.O.~(1956a). The theory of search, III The optimum distribution of searching efforts. \emph{Operations Research}~4, 613--626. 

\item{} Kratzke, T.M., Stone, L.D., and Frost J.R.~(2010). Search and rescue optimal planning system.   \emph{Proceedings of the 13th International Conference on Information Fusion, Edinburgh, UK, July 2010}, 26--29.

\item{} Lidbetter, T.~(2020). Search and rescue in the face of uncertain threats. \emph{European Journal of Operations Research}~285, 1153--1160.

\item{} Lin, K.Y.~(2021). Optimal patrol of a perimeter. \emph{Operations Research}~70(5), 2860--2866.

\item{} Richardson, H.R.~and Stone, L.D.~(1971). Operations analysis during the underwater search for Scorpion. \emph{Naval Research Logistic Quarterly}~18, 141--157. 

\item{} Richardson, H.R~and Discenza, J.H.~(1980). The United States Coast Guard Computer-assisted Search Planning System (CASP). \emph{Naval Research Logistic Quarterly}~27,  659--680. 


\item{} Stone, L.D.~(1973). Totally optimality of incrementally optimal allocations. \emph{Naval Research Logistics Quarterly}~20, 419--430.

\item{} Stone, L.D.~(1975). \emph{Theory of Optimal Search}.  Academic Press: New York.

\item{} Stone, L.D.~(1976). Incremental and total optimization of separable functionals with constraints. \emph{SIAm Journal on Control and Optimization}~14,  791--802.

\item{} Stone, L.D., Royset, J.O., and Washburn, A.R.~(2016). \emph{Optimal Search for Moving Targets}. Springer: New York.

\item{} Stone, L.D.~and Stanshine J.A.~(1971). Optimal searching using uninterrupted contact investigation.  \emph{SIAM Journal of Applied Mathematics}~20, 241--263.

\item{} Stone, L.D.~(1992). Search for the SS Central America: mathematical treasure hunting. \emph{Interfaces}~22: 32--54. 

\item{} Stone, L.D., Keller, C.M. , Kratzke, t.M.~and Strumpfer, J.P.~(2014). Search for the wreckage of Air France AF 447. \emph{Statistical Science}~29:69--80. 

Artech House: Boston,


\item{} Vermeulen, J.F.~and Brink, M.V.~(2017).   The search for an altered moving target.  \emph{Journal of the Operational Research Society}~56(5), 514--525.

\item{} Washburn, A.~(2014). \emph{Search and Detection}, 5th Edition. Create Space: North Carolina.


\end{description}

\end{document}